\documentclass[12pt]{amsart}
\usepackage{txfonts}
\usepackage{amscd,amssymb}
\usepackage[colorlinks,plainpages,backref,urlcolor=blue]{hyperref}
\usepackage{verbatim}
\usepackage{xcolor}

\newtheorem{theorem}{Theorem}[section]
\newtheorem{corollary}[theorem]{Corollary}
\newtheorem{lemma}[theorem]{Lemma}
\newtheorem{prop}[theorem]{Proposition}

\theoremstyle{definition}
\newtheorem{definition}[theorem]{Definition}

\newtheorem{remark}[theorem]{Remark}

\newcommand{\C}{\mathbb{C}}

\newcommand{\PP}{\mathbb{P}}

\DeclareMathAlphabet{\pazocal}{OMS}{zplm}{m}{n}

\newcommand{\OO}{{\pazocal O}}

\newcommand{\cE}{{\mathcal{E}}}

\DeclareMathOperator{\im}{Im}
\DeclareMathOperator{\Coker}{Coker}
\DeclareMathOperator{\Ker}{Ker}

\def\dot{\mathchar"013A}
\newcommand{\hdot}{{\raise1pt\hbox to0.35em{\Huge $\dot$}}}

\begin{document}
\date{\today}

\title[Addition-deletion results for plus-one generated curves]%
{Addition-deletion results for plus-one generated curves}
\author{Anca~M\u acinic$^*$ and Piotr Pokora$^{**}$}
\thanks{$^*$ Partially supported by 
 a grant of the Romanian Ministry of Education and Research, CNCS - UEFISCDI, project number \textbf{PN-III-P4-ID-PCE-2020-2798}, within PNCDI III}
\thanks{$^{**}$ Partially supported by  The Excellent Small Working Groups Programme \textbf{DNWZ.711/IDUB/ESWG/2023/01/00002} at the University of the National Education Commission Krakow.}

\maketitle
\thispagestyle{empty}
\begin{abstract}
In the recent paper A. Dimca proves that when one adds to or deletes a line from a free curve the resulting curve is either free or plus-one generated. We prove the converse statements, we give an additional insight into the original deletion result, and derive a characterisation of free curves in terms of behaviour to addition/deletion of lines.
We describe the possible splitting types of the bundle of logarithmic vector fields associated to a plus-one generated curve.
\end{abstract} 

\section{Introduction}

Let $\mathcal{C}$ be a reduced curve in $\mathbb{P}^2 \coloneqq \mathbb{P}^{2}_{\C}$ defined by the equation $f_{\mathcal{C}}=0$, where $f_{\mathcal{C}} \in S \coloneqq \C[x,y,z]$.
The logarithmic derivation module associated to the curve $\mathcal{C}$ is the submodule of ${\rm Der}(S)$ defined by
$$D(\mathcal{C}) \coloneqq \{\theta \in {\rm Der}(S) \; | \; \theta(f_{\mathcal{C}}) \in S f_{\mathcal{C}} \}.$$ The Euler derivation $\theta_E = x \partial x + y \partial y + z \partial z$ is an element of  $D(\mathcal{C})$ for any curve $\mathcal{C}$. Furthermore, the module $D(\mathcal{C})$ admits a decomposition as the direct sum
$$D(\mathcal{C}) = D_0(\mathcal{C}) \oplus S \theta_E,$$
where $$D_0(\mathcal{C}) \coloneqq \{\theta \in {\rm Der}(S) \; | \; \theta(f_{\mathcal{C}})=0 \}.$$

The curve $\mathcal{C}$ is called {\it free} if $D_0(\mathcal{C})$ is free as an S-module. In this case, $D_0(\mathcal{C})$ is a rank $2$ graded module and we call {\it exponents} of $\mathcal{C}$ the pair of degrees of the elements in a basis of homogeneous derivations of $D_0(\mathcal{C})$. In other words, if $\{\theta_2, \theta_3\}$ is a basis of homogeneous derivations of $D_0(\mathcal{C})$, then $\exp(\mathcal{C}) \coloneqq (\deg(\theta_2), \deg(\theta_3))$.

We recall the definition of a recent notion introduced by Abe in \cite{A0} for hyperplane arrangements and extended to curves by Dimca and Sticlaru in \cite{DS0}.

\begin{definition}
\label{def:POG}
A reduced curve $\mathcal{C}$ in $\mathbb{P}^2$ is called {\it plus-one generated} if its associated module of derivations $D(\mathcal{C})$ admits a minimal set of homogeneous generators $\{\theta_E, \theta_2, \theta_3, \psi\}$ that satisfy a unique relation of the form
\begin{equation}
\label{eq:POGrel}
f_1 \theta_E + f_2 \theta_2 + f_3 \theta_3+ \alpha \psi = 0,
\end{equation}
where $f_i \in S$ for $i =  1, 2, 3$ and $\alpha \in S$ with $\deg(\alpha) =1$. The {\it exponents} of $\mathcal{C}$ are, by definition, the pair $(\deg(\theta_2), \deg(\theta_3))$ and $\deg(\psi)$ is called the {\it level} of $\mathcal{C}$.
\end{definition}

In the recent paper \cite{Dimca}, Dimca proves the following addition result for free curves (see Section \ref{sec:prelim} for the additional notions and notations).

\begin{theorem}
\label{thm:add}
Let $\mathcal{C}$ be a free reduced curve in $\mathbb{P}^2$ of exponents $(d_2, d_3)$ and $L \subset \mathbb{P}^2$ be a line that is not an irreducible component of $\mathcal{C}$. Then $\mathcal{C} \cup L$ is either free with the exponents  $(d_2, d_3+1)$  and $d_2+1 = | \mathcal{C} \cap L|+ \epsilon(\mathcal{C}, L) $
or plus-one generated with the exponents $(d_2+1, d_3+1)$ and level $d = | \mathcal{C} \cap L|+ \epsilon(\mathcal{C}, L) -1$.
\end{theorem}

In Section \ref{sec:prelim} we recall some necessary notions and results from the literature and we prove some new results related to Chern numbers and splitting types for the vector bundle associated to a plus-one generated curve, defined in  \ref{sec:prelim}. In general, evaluating splitting types onto projective lines for vector bundles is a difficult endeavor. 
We describe in Theorem \ref{thm:splitting_type_values} the set of possible values for the splitting types of the vector bundle associated to a plus-one generated curve, generalizing similar results on plus-one generated line arrangements, respectively on nearly free curves, from \cite{AIM}, respectively from \cite{AD}.\\

In Section \ref{sec:main} we present our main results.
We prove in Theorem \ref{thm:add_converse} a converse of the above result. In the same paper Dimca proved a deletion result for free reduced curves, see \cite[Theorem 1.1]{Dimca}. It states that for a free curve $\mathcal{C}$ with a line $L$ being an irreducible component, the curve $\mathcal{C}'$ obtained by removing from $\mathcal{C}$ line $L$ is either free or plus-one generated. The author asks whether some relation between the exponents of $\mathcal{C}$ and the exponents (and level) of $\mathcal{C}'$ can be given. We answer this question in Theorem \ref{thm:exp_computation}.
We prove in Theorem \ref{thm:del_inv} a converse of the deletion result \cite[Theorem 1.1]{Dimca}.

 Theorem \ref{thm_add_del_converse_free} generalizes a similar result on conic-line arrangements of Schenck and Tohaneanu, see \cite[Theorem 2.5]{ST}, describing the situation when the addition or the deletion of a projective line from a free curve results in a free curve.

Finally, summing up Theorems  \ref{thm:add_converse},  \ref{thm:exp_computation}, \ref{thm:del_inv},  \ref{thm_add_del_converse_free} and \cite[Theorems 1.1, 1.2]{Dimca}, we give in Theorems \ref{thm:char_free_addition} and \ref{thm:char_free_deletion} a freeness criterion for curves in terms of behaviour to addition or to deletion of projective lines.

\section*{Acknowledgments}
The first author is grateful to Jean Vall\`es for patiently answering many questions regarding Chern polynomials.

\section{Preliminaries}
\label{sec:prelim}

Unless otherwise stated, $\mathcal{C}$ is a reduced curve in $\mathbb{P}^2$ and $L \subset \mathbb{P}^2$ is a line that is an irreducible component of $\mathcal{C}$. Let $f_{\mathcal{C}} \in S$ be a homogeneous polynomial such that  $f_{\mathcal{C}} = 0$ defines $\mathcal{C}$.
We denote by  $(\mathcal{C}, \mathcal{C}', \mathcal{C}'')$ the triple obtained by the deletion-restriction with respect to the line $L$, i.e., $\mathcal{C}'$ is the curve obtained by removing from $\mathcal{C}$ line $L$, and $\mathcal{C}'' = \mathcal{C}' \cap L$.
Without loss of generality, we may take $\alpha_L = z$ as the defining linear form of the line $L$. Then $f_{\mathcal{C}} = zf_{\mathcal{C}'}$.\\

Denote by $\cE_{\mathcal{C}}$ the locally free sheaf (i.e. vector bundle) obtained by the sheafification of the graded module $D_0(\mathcal{C})$.
 Recall that $\cE_{\mathcal{C}} = \mathcal{T}_{\mathcal{C}}(-1)$, where $\mathcal{T}_{\mathcal{C}}$ is Saito's sheaf of logarithmic vector ﬁelds along $\mathcal{C}$ (see \cite{S}).

\begin{remark}
\label{rem:quasi}
By \cite[Theorem 2.3]{Dimca} (see also \cite[Theorem 1.6, Remark 1.8]{STY}), we have an exact sequence of sheaves (actually vector bundles) for the triple $(\mathcal{C}, \mathcal{C}', \mathcal{C}'')$ defined above, namely
\begin{equation}
\label{eq:sheaves_ex_seq_general}   
0 \rightarrow \mathcal{E}_{\mathcal{C}'}(-1) \overset{\alpha_L}\longrightarrow  \mathcal{E}_{\mathcal{C}} \rightarrow \mathcal{O}_L(1-|\mathcal{C}''| -\epsilon(\mathcal{C}',L)) \rightarrow 0,
\end{equation}
where $\epsilon(\mathcal{C}',L)\in \mathbb{Z}$ is a constant defined by Dimca in \cite[page 2]{Dimca}. This constant measures the defect from the property of being a quasi-homogeneous triple, namely  if all singularities of the triple $(\mathcal{C}, \mathcal{C}', \mathcal{C}'')$ are quasi-homogeneous, 
then the above exact sequence has becomes
\begin{equation}
\label{eq:sheaves_ex_seq}
0 \rightarrow \mathcal{E}_{\mathcal{C}'}(-1) \overset{\alpha_L}\longrightarrow  \mathcal{E}_{\mathcal{C}} \rightarrow \mathcal{O}_L(1-|\mathcal{C}''|) \rightarrow 0.
\end{equation}
\end{remark}

\noindent Note also that by \cite[Remark 2.4]{Dimca} we only need to assume that the singularities of $\mathcal{C}$ on $L$ are quasi-homogeneous in order for the exact sequence \eqref{eq:sheaves_ex_seq} to hold.\\

Denote by $c(\mathcal{E})$ the Chern polynomial of a complex vector bundle $\mathcal{E}$. It is well-known that, for an exact sequence of vector bundles
$$ 
0 \rightarrow \mathcal{E} \rightarrow \mathcal{F}  \rightarrow \mathcal{G}  \rightarrow 0,
$$
we have the relation 
\begin{equation}
\label{eq:chern_poly_product}
c(\mathcal{F}) = c(\mathcal{E}) c(\mathcal{G}).
\end{equation}
In particular, if $c_i(\mathcal{E})$ is the $i$-th Chern number of $\mathcal{E}$, then we have the equality
\begin{equation}
\label{eq:c_2}
c_2(\mathcal{F}) = c_2(\mathcal{E}) + c_1(\mathcal{E})c_1(\mathcal{G}) + c_2(\mathcal{G}).
\end{equation}

\begin{lemma}
\label{lemma: E(-1)} Let $\mathcal{E}$ be a rank $2$ vector bundle and $c(\mathcal{E}) = 1+c_1(\mathcal{E})t+c_2(\mathcal{E})t^2$ its Chern polynomial. Then the Chern polynomial of $\, \mathcal{E}(-1)$ has the following form
$$ c(\mathcal{E}(-1)) = 1+(-2+c_1(\mathcal{E}))t+(c_2(\mathcal{E})- c_1(\mathcal{E}) +1)t^2.$$
\end{lemma}

\begin{proof}
 Immediate follows from the equality $\mathcal{E}(-1) = \mathcal{E} \otimes_{\mathcal{O}_{\mathbb{P}^2}} \mathcal{O}_{\mathbb{P}^2}(-1)$ and the formula for the Chern classes of a tensor product of a rank $2$ vector bundle and a line bundle, taking into account that $c(\mathcal{O}_{\mathbb{P}^2}(-1)) = 1-t$.   
\end{proof}

Let us recall a result due to Yoshinaga devoted to the splitting type along a line of rank $2$ vector bundles on $\PP^2$, which we will apply to sheaves of logarithmic vector fields along a curve, see Remark \ref{rem:sections}.

\begin{theorem}[\cite{Y}]
\label{thm:VB_split}
Let $\mathcal{E}$ be a rank $2$ vector bundle on $\PP^2$, $L$ a line in $\PP^2$ and $\mathcal{E}|_{L} =\OO_L(-a) \oplus  \OO_L(-b)$. Then
$$
c_2(\mathcal{E}) - a b = {\rm dim}_{\C}Coker(\Gamma_*(\mathcal{E}) \overset{\pi_L}{\longrightarrow}(\Gamma_*(\mathcal{E}|_{L})),
$$
where $c_2(\mathcal{E})$ is the second Chern number of $\mathcal{E}$ and $\pi_L$ is the morphism of graded modules of sections induced by the restriction to $L$ of the vector bundle $\mathcal{E}$. Moreover, $\mathcal{E}$ is splitting if and only if $\pi_L$ is surjective. 
\end{theorem}

The pair $(a,b)$ is called the {\it splitting type} onto $L$ of the bundle $\mathcal{E}$.

\begin{remark}
\label{rem:sections}
 For an arbitrary reduced curve $\mathcal{C}$ we have an isomorphism $\Gamma_*(\mathcal{E}_{\mathcal{C}}) \cong D_0(\mathcal{C})$, see \cite[Proposition 2.1]{AD} for details. Then $\mathcal{C}$ is free if and only if $\mathcal{E}_{\mathcal{C}}$ is splitting.
\end{remark}

Another important instrument to test the freeness property for hyperplane arrangements is Saito's criterion. A version of this result holds for reduced hypersurfaces in $\mathbb{P}^{n}
_{\mathbb{C}}$, in particular for reduced curves in $\mathbb{P}^{2}$, see \cite{ST}.

\begin{theorem}[Saito's criterion for curves]
\label{thm:saito}
Let $\mathcal{C}$ be a reduced curve in $\mathbb{P}^2$ defined by $f_{\mathcal{C}}=0$ with $f_{\mathcal{C}} \in S$. Given a set of homogeneous derivations $\{\theta_1, \theta_2, \theta_3\} \subset D(\mathcal{C})$, the following are equivalent:
\begin{enumerate}
\item  $\{\theta_1, \theta_2, \theta_3\}$ is a basis in  $D(\mathcal{C})$.
\item The determinant of the matrix $M$ is equal to $f_{\mathcal{C}}$ (modulo a non-zero constant), where $M$ is the $3 \times 3$-matrix with lines $[\theta_i(x) \; \theta_i(y) \; \theta_i(z)], \; i=1, 2, 3$. 
\end{enumerate} 
\end{theorem}

We will denote by $mdr(\mathcal{C}) = mdr(f_{\mathcal{C}})$ the minimal degree of relations for $f_{\mathcal{C}}$, which is by definition the smallest degree of a non-zero derivation from $D_0(\mathcal{C})$. It is known that for free or for plus-one generated curves $\mathcal{C}$ one has $mdr(\mathcal{C}) = {\rm min} \exp(\mathcal{C}) = {\rm min} \{d_{2},d_{3}\}$.

\begin{prop}
\label{prop:chern_2}
Let $\mathcal{C}$ be a plus-one generated curve with the exponents $(d_2, d_3)$ and level $d$. Then the second Chern number of $\cE_{\mathcal{C}}$ is given by the formula:
\begin{equation}
\label{eq:chern}
c_2(\cE_{\mathcal{C}}) = d_2(d_3-1)+d-d_3+1
\end{equation}
\end{prop}

\begin{proof}
Assume without loss of generality that $d_2 \leq d_3$. Since $d_2+d_3= \deg(\mathcal{C})$, then  $d_2 = mdr(\mathcal{C}) \leq \deg(\mathcal{C})/2$. Let $L_0$ be a generic line in $\mathbb{P}^2$ and denote by $(d_2^{L_0}, d_3^{L_0}), \; d_2^{L_0} \leq d_3^{L_0},$ the splitting type onto $L_0$ of the vector bundle $\cE_{\mathcal{C}}$. Recall that $d_2^{L_0} + d_3^{L_0} =  \deg(\mathcal{C}) -1$. From \cite{AD}, we know that 
\begin{equation}
\label{eq:c_2_int}
c_2(\cE_{\mathcal{C}}) = d_2^{L_0} d_3^{L_0} + \nu(\mathcal{C}),
\end{equation}
where $\nu(\mathcal{C})$ is the defect of the curve $\mathcal{C}$, which is equal to $(d - d_3 +1)$ by \cite[Corollary 3.10]{DS0}. In order to finish our proof, it remains to connect the generic splitting type $(d_2^{L_0}, d_3^{L_0})$  with the exponents of $\mathcal{C}$.
 By \cite[Proposition 3.2 (1)]{AD}, $d_2^{L_0}  \leq mdr(\mathcal{C})$, i.e., $d_2^{L_0} \leq d_2$. There are, according to  \cite[Proposition 3.2 (2)-(4)]{AD}, three cases to consider, namely: \\

Case 1: $d_2^{L_0} < (\deg(\mathcal{C}) -2) / 2$, and then $d_2^{L_0}  = mdr(\mathcal{C})$. \\

Case 2:  $d_2^{L_0} \geq  (\deg(\mathcal{C}) -2) / 2$ and $\deg(\mathcal{C}) = 2m$. Then $d_2^{L_0}=m-1$, which implies that $mdr(\mathcal{C}) \in \{m-1, m\}$.\\

Case 3:  $d_2^{L_0} \geq  (\deg(\mathcal{C}) -2) / 2$ and $\deg(\mathcal{C}) = 2m+1$. Then $d_2^{L_0}=m$, so $m \leq mdr(\mathcal{C}) \leq (2m+1)/2$, i.e., $mdr(\mathcal{C}) = m$.\\

In all the above cases we get a formula for $d_2^{L_0}$ in terms of $mdr(\mathcal{C})$. In Case 1 and Case 3 we get $d_2^{L_0} = mdr(\mathcal{C})$. In Case 2 we get either $d_2^{L_0} = mdr(\mathcal{C}) = m-1$ or $d_2^{L_0} = mdr(\mathcal{C})-1 = m-1$.
Since $mdr(\mathcal{C}) = d_2, \; d_2+d_3= \deg(\mathcal{C})$, and $d_2^{L_0} + d_3^{L_0} = \deg(\mathcal{C}) - 1$, in all the above cases formula \eqref{eq:chern} follows from \eqref{eq:c_2_int}.
\end{proof}

\begin{corollary}
\label{cor:generic_spl_type}
Let $\mathcal{C}$ be a plus-one generated curve with exponents $(d_2, d_3), \; d_2 \leq d_3,$ and level $d$.  Let $(d_2^{L_0}, d_3^{L_0}),  d_2^{L_0} \leq  d_3^{L_0},$ be  the generic splitting type of $\cE_{\mathcal{C}}$. Then  $d^{L_0}_2 = d_2$, except possibly when $\deg(\mathcal{C}) = 2m$, when  $d_2^{L_0} \in  \{d_2-1, d_2\}$.
\end{corollary}

\begin{proof}
Immediate from the proof of Proposition \ref{prop:chern_2}.
\end{proof}

In fact, we are able to describe all possible splitting types of $\cE_{\mathcal{C}}$ onto an arbitrary line $L \subset \mathbb{P}^2$, if $\mathcal{C}$ is a plus-one generated curve.

\begin{theorem}
\label{thm:splitting_type_values}
Let $\mathcal{C}$ be a plus-one generated curve with exponents $(d_2, d_3)$ and level $d$.  Then the splitting type of $\cE_{\mathcal{C}}$ onto an arbitrary line $L \subset \mathbb{P}^2$  is one of the pairs  $\{(d_2, d_3-1), (d_2-1, d_3), (d_2-(d -d_3+1), d)\}$.
\end{theorem}

\begin{proof}
The proof is essentially the same as the one of \cite[Theorem 4.4]{AIM}, which proves a similar result for plus-one generated projective line arrangements. We reproduce the arguments here for completion. 

Assume without loss of generality that $d_2 \leq d_3$.
Denote by $(d^L_2, d^L_3), \; d^L_2 \leq  d^L_3$ the the splitting type of $\cE_{\mathcal{C}}$ onto the line $L $. Then $d^L_2 + d^L_3 = \deg(\mathcal{C}) - 1$, see for instance  \cite[Proposition 3.1]{AD}. 

Since $\nu(\mathcal{C})=d - d_3 +1$ (\cite[Proposition 3.7]{DS0}), by 
 \cite[Corollary 3.4]{AD} we have that 
 
 $$(d^L_2, d^L_3) \in \{(d_2, d_3-1), (d_2-1, d_3), \dots,  d_2-(d -d_3+1, d)\}.$$ 
 
 If $d^L_2 \in \{d_2, d_2-1\}$ we are done, since $d_2 + d_3 = \deg(\mathcal{C})$.
  It remains to consider the case $d^L_2 < d_2-1$. By Remark \ref{rem:sections}, the domain of the map $\pi_L$ from Theorem \ref{thm:VB_split} can be identified to the module of derivations $D_0(\mathcal{C})$. One can choose the minimal system of generators $\{\theta_E, \theta_2, \theta_3, \psi\}$ of $D(\mathcal{C})$ from Definition \ref{def:POG} such that $\{ \theta_2, \theta_3, \psi\}$  is a minimal set of generators for $D_0(\mathcal{C})$, with $\deg(\theta_2) = d_2, \; \deg(\theta_3) = d_3, \; \deg(\psi) = d$. 
  
Let  $\{ \theta_2^L, \theta_3^L \}$ be a homogeneous basis of  the codomain of $\pi_L$,  $\Gamma_*(\mathcal{E}_{\mathcal{C}}|_{L})$, which is a free graded module of rank $2$ over $\overline{S}\coloneqq S / (\alpha_L)$, with  $\deg( \theta_2^L)  = d^L_2$ and $\deg( \theta_3^L)  = d^L_3$.
Since $d^L_2 < d_2-1$,  we get $d^L_3 > d_3$, hence 
 $\pi_L(\theta_2), \pi_L(\theta_3)$ are in the $\overline{S}$, a submodule of $\Gamma_*(\mathcal{E}_{\mathcal{C}}|_{L})$ generated by $\theta_2^L$.  
 
Let $\pi_L(\psi) = f \theta^L_2  + g \theta^L_3$, for some  $f,g \in \overline{S}$. Then necessarily $g \neq 0$, otherwise we would have $\dim_{\C} \Coker (\pi_L) = \infty$, a contradiction with Theorem \ref{thm:VB_split}. Moreover, we must have $\deg(g) = 0$, otherwise we can construct an infinite series of elements from $\Gamma_*(\mathcal{E}_{\mathcal{C}}|_{L})$ which are not in the image of $\pi_L$, for instance elements of type $m \theta^L_2+\theta^L_3$, with $m \in \overline{S}$ arbitrary monomial, again a contradiction to the finite dimensionality over $\C$  of  $\Coker (\pi_L)$, stated in Theorem \ref{thm:VB_split}. 
 
In conclusion, $g \in \C \setminus \{0\}$, i.e.,  $\deg \theta^L_3 = d$.
\end{proof}

\begin{remark}
\label{rem:first_chern_number}
If $\mathcal{C}$ is a reduced projective curve, then the first Chern number of $\cE_{\mathcal{C}}$ is given by the formula $c_1(\cE_{\mathcal{C}}) = 1 -\deg(\mathcal{C})$, see for instance \cite{AD}.
\end{remark}

\begin{prop}
\label{prop:deriv_seq}
Let $\mathcal{C}, \mathcal{C}'$ be reduced curves in $\mathbb{P}^2$ such that  $\mathcal{C} = \mathcal{C}' \cup \{L\}$, where $L \subset \mathbb{P}^2$ is a line not included in $\mathcal{C'}$ having the equation $\alpha_L=0$. Let $\mathcal{C}'' = \mathcal{C}' \cap L$.
Then we have the exact sequence of derivation modules, where $\rho$ is defined by taking modulo $\alpha_L$:
\begin{equation}
\label{eq:ex_seq}
0 \rightarrow D(\mathcal{C}')   \overset{\cdot \alpha_L}\rightarrow D(\mathcal{C})   \overset{\rho}\rightarrow D(\mathcal{C}")
\end{equation}
\end{prop}

\begin{proof}
Let $f_{\mathcal{C}} \in \C[x,y,z]$ be the defining polynomial of $\mathcal{C}$. We may assume that $\alpha_L = z$. Then $f_{\mathcal{C}} = zf$ with $f \in \C[x,y,z]$, where $f$ is the defining polynomial for $\mathcal{C}'$. It is easy to check that if $\theta' \in  D(\mathcal{C}')$, then $z \theta' \in  D(\mathcal{C})$. Also, notice that the multiplication map by $z$ is obviously injective.\\

\noindent Let $\theta \in D(\mathcal{C})$. Then $f_{\mathcal{C}}$ divides $\theta(f_{\mathcal{C}})$. This implies, since ${\rm gcd}(z,f)=1$, that $z$ divides $\theta(z)$ and $f$ divides $\theta(f)$. Let
 $$\theta = f_1 \partial_x + f_2 \partial_y + f_3 \partial_z \in \Ker(\rho).$$
 Since $z$ divides $\theta(z)$, we get that $z\, | \,f_3 $. Now from $$\rho(\theta) = f_1(x,y,0) \partial_x + f_2(x,y,0) \partial_y = 0$$ it follows that  $z\, | f_1$ and $z\, | f_2 $. Then $\theta = z 
 \theta'$. Since  $f$ divides $\theta(f) = z \theta'(f)$ and ${\rm gcd}(z,f) = 1$, it follows that $f$ divides $\theta'(f)$, so $\theta' \in D(\mathcal{C}')$. In conclusion, we obtain $\im(\cdot z) = \im(\cdot \alpha_L) \supset \Ker(\rho)$. The converse inclusion is obvious, so we have $\im(\cdot \alpha_L) = \Ker(\rho)$.\\

Let us check that the map $\rho$ is well-defined, i.e., for $\theta \in D(\mathcal{C})$ we have $\rho(\theta) \in D(\mathcal{C}'')$.
For $X \in \mathcal{C}''$ having a defining equation $\gamma_X =0, \; \gamma_X \in \mathbb{C}[x,y]$, we have that $\gamma_X$ divides $f(x,y,0)$. Let $k \in \mathbb{N}_{\geq 1}$ be maximal such that 
$$f(x,y,0) = \gamma_X^k h, \; h \in \C[x,y], \; {\rm gcd}(\gamma_X, h)=1.$$
Evaluating at $z=0$ the expression $f\, | \,\theta(f)$ one gets $f(x,y,0) \, | \, \rho(\theta)(f(x,y,0))$. 
Denote furthermore $\overline{\theta}:= \rho(\theta)$. We have the divisibility  $$\gamma_X^k h \, | \, \overline{\theta}(\gamma_X^k h)$$ and we need to show that $$\gamma_X \, | \, \overline{\theta}(\gamma_X).$$ 
Observe that $\overline{\theta}(\gamma_X^k h) = \gamma_X^k \overline{\theta}(h) + h \overline{\theta}(\gamma_X^k)$ and since ${\rm gcd}(\gamma_X, h)=1$, one has $\gamma_X^k \, | \, \overline{\theta}(\gamma_X^k)$. If $k=1$, then we are done. Otherwise, if $k>1$, we compute $\overline{\theta}(\gamma_X^k) = k \gamma_X^{k-1} \overline{\theta}(\gamma_X)$ to conclude that $\gamma_X\, | \, \overline{\theta}(\gamma_X)$. Since this holds for an arbitrary $X \in \mathcal{C}''$, it follows $\overline{\theta} \in D(\mathcal{C}'')$.
\end{proof}

\begin{remark}
\label{rem:pi_L_restr}
If one takes $\mathcal{E}$ from Theorem \ref{thm:VB_split} to be sheaf $\mathcal{E}_{\mathcal{C}}$ of logarithmic vector fields along $\mathcal{C}$, which is a locally free sheaf, i.e., a vector bundle by \cite{S}, then $\pi_L$ is actually the restriction to $D_0(\mathcal{C})$ of the map $\rho$ from \eqref{eq:ex_seq} via the identification of $\Gamma_*(\widetilde{D_0(\mathcal{C})})$ to $D_{0}(\mathcal{C})$, recalled in Remark \ref{rem:sections}.
\end{remark}

\section{Main results}
\label{sec:main}

\begin{theorem}
\label{thm:add_converse}
Let $\mathcal{C}$ be a plus-one generated reduced curve in $\mathbb{P}^2$ with the exponents $(d_2, d_3)$ and level $d$ such that there exists a line $L \subset \mathbb{P}^2$ that is an irreducible component of $\mathcal{C}$, $\mathcal{C} =  \mathcal{C}' \cup \{L\}$, $L$ is not included in $\mathcal{C}'$, and $|\mathcal{C}''| + \epsilon(\mathcal{C}', L)= d+1$, where $\mathcal{C}'' \coloneqq \mathcal{C}' \cap L$. Then $\mathcal{C}'$ is free with the exponents $(d_2-1, d_3-1)$.
\end{theorem}

\begin{proof}

 A computation of the second Chern number of $\mathcal{E}_{\mathcal{C}'}$, using \eqref{eq:c_2}
and \eqref{eq:sheaves_ex_seq_general} and taking into account that $d+1=|\mathcal{C}''| + \epsilon(\mathcal{C}', L)$, shows that 
$$c_2(\mathcal{E}_{\mathcal{C}'}) = (d_2-1) (d_3-1)$$

To prove that $\mathcal{C}'$ is free, it would be enough to show that there is a line $K$ in $\mathbb{P}^2$ such that the splitting type of $\mathcal{E}_{\mathcal{C}'}$ onto $K$ is equal to $(d_2-1, d_3-1)$. We will show that $(d_2-1, d_3-1)$ is the generic splitting type of $\mathcal{E}_{\mathcal{C}'}$ onto a line. Let $K$ be a generic line, with corresponding splitting type $(d_2^K, d_3^K), \; d_2^K \leq d_3^K$.

We may assume without loss of generality that $d_2 \leq d_3$.
Since $d_2^K + d_3^K = d_2 + d_3 -2$ and $(d_2-1) (d_3-1) \geq d_2^K  d_3^K$ by Theorem \ref{thm:VB_split}, it follows that
$$d_2^K \leq d_2 - 1$$
From \cite[Proposition 3.2 (2)-(4)]{AD}, we know that $d_2^K = mdr(\mathcal{C}')$ except possibly when $\deg(\mathcal{C}')=2m$ and  $d_2^K = m-1$ or  $\deg(\mathcal{C}')=2m+1$ and $d_2^K = m$. We will treat separately these cases.\\

\noindent First assume $d_2^K = mdr(\mathcal{C}')$. By \cite[Proposition 3.1]{DIS}, $mdr(\mathcal{C}') \in \{d_2-1, d_2 \}$. If $mdr(\mathcal{C}')  = d_2-1$ then $d_2^K = d_2-1$, so $(d_2^K, d_3^K) = (d_2-1, d_3-1)$, which proves the claim.

\noindent If $mdr(\mathcal{C}')  = d_2$, then $d_2^K = d_2$, but this contradicts the inequality $d_2^K \leq d_2 - 1$, so the case  $mdr(\mathcal{C}')  = d_2$ cannot happen.\\

Case $\deg(\mathcal{C}')=2m, \; d_2^K = m-1$. Then $m \leq d_2$, but $\deg(\mathcal{C}) = 2m+1$, so $d_2 = m$. It follows $(d_2^K, d_3^K) = (d_2-1, d_3-1)$, which proves the claim.\\

 Case $\deg(\mathcal{C}')=2m+1, \; d_2^K = m$. Then $m+1 \leq d_2$,  but $\deg(\mathcal{C}) = 2m+2$, so $d_2 = m+1$. Again it follows that $(d_2^K, d_3^K) = (d_2-1, d_3-1)$, which completes our proof.\\
\end{proof}
It has come to our attention that the next result is also proved in a more recent version of \cite{Dimca}, but with a different proof.
\begin{theorem}
\label{thm:exp_computation}
Let $\mathcal{C}$ be a reduced free curve with the exponents $(d_2, d_3)$, not necessarily ordered, such that $\mathcal{C} = \mathcal{C}' \cup L$ for some line $L \subset \mathbb{P}^2, \; L \not\subset \mathcal{C}'$. Let $\mathcal{C}'' \coloneqq  \mathcal{C}' \cap L$.
\begin{enumerate}
    \item[1.] If $\mathcal{C}'$ is free, then $\exp(\mathcal{C}') = (d_2, d_3-1)$ and $|\mathcal{C}''|  + \epsilon(\mathcal{C}',L)= d_2+1$.
    \item[2.] If $\mathcal{C}'$ is plus-one generated, then $\exp(\mathcal{C}') = (d_2, d_3)$ and the level $d= \deg(\mathcal{C}') - | \mathcal{C}''| - \epsilon(\mathcal{C}',L)$.
\end{enumerate}
\end{theorem}

\begin{proof}
Assume first $\mathcal{C}'$ is free with the exponents $(d'_2, d'_3), \, d'_2 \leq d'_3$. Notice that $d_2+d_3 = d'_2+d'_3+1$
and recall that by \cite[Proposition 3.1]{DIS}, $mdr(\mathcal{C}) = \min \{d_2, d_3\} \in \{d'_2, d'_2+1\}$. It follows that 
$$(d'_2, d'_3) \in \{(d_2-1, d_3), \; (d_2, d_3-1)\}.$$ 
Assume without loss of generality that $(d'_2, d'_3)=(d_2, d_3-1)$, since $(d_2, d_3)$ are unordered. Then the Chern polynomial has the following form
$$c(\mathcal{E}_{\mathcal{C}'}) = 1 + (-d'_2 - d'_3)t + (d'_2 d'_3) t^2 = 1 + (-d_2 +1 - d_3)t + d_2(d_3-1) t^2$$ and by Lemma \ref{lemma: E(-1)}, 
$$c(\mathcal{E}_{\mathcal{C}'}(-1)) = 1 + (-1-d_2 - d_3)t + (d_2 d_3 + d_3) t^2.$$ Since $\mathcal{C}$ is free with the exponents $(d_2, d_3), \; c(\mathcal{E}_{\mathcal{C}}) = 1+ (-d_2-d_3)t + d_2d_3 t^2$. 
Compute \eqref{eq:c_2}
for the exact sequence \eqref{eq:sheaves_ex_seq_general} to conclude $|\mathcal{C}''| + \epsilon(\mathcal{C}',L) = d_2+1$. \\

Assume next that $\mathcal{C}'$ is plus-one generated with $\exp(\mathcal{C}') = (d'_2, d'_3), \, d'_2 \leq d'_3,$ and level $d$. Choose an order among the exponents of $\mathcal{C}$, say $d_2 \leq d_3$.

 We have $d_2+d_3 = d'_2+d'_3$ and, by \cite[Proposition 3.1]{DIS}, $mdr(\mathcal{C}) = d_2 \in \{d'_2, d'_2+1\}$. Assume that $d_2 = d'_2 +1$. Since $f_{\mathcal{C}'} \;  | \; f_\mathcal{C}$,  we have an inclusion between modules of derivations:
$$D(\mathcal{C}) \subset D(\mathcal{C}')$$
Since $\mathcal{C}$ is free with the exponents $(d_2, d_3)$, the module $D(\mathcal{C})$ is spanned by $\theta_E$ and two other homogeneous derivations $\theta_2, \theta_3$ of degrees $d_2 = d'_2+1$, respectively $d_3=d'_3-1$. The module of derivations $D(\mathcal{C}')$ of the plus-one generated curve $\mathcal{C}'$ is spanned by  $\theta_E$ and three other homogeneous derivations $\theta'_2, \theta'_3, \psi$ of degrees $d'_2, \; d'_3$,  respectively $d$. Due to degree constraints, both $\theta_2, \theta_3$ are in the submodule of $D(\mathcal{C}')$ spanned by $ \theta_E$ and  $\theta'_2$, but this implies that $\theta_2, \theta_3, \theta_E$ are not S-independent, contradiction. 
Then necessarily $d_2 = d'_2$, so 
$$\exp(\mathcal{C}') = (d_2, d_3)$$ It remains to prove the formula for the level $d$.

By Proposition \ref{prop:chern_2}, 
$$c(\mathcal{E}_{\mathcal{C}'}) = 1 + (1-d_2 - d_3)t + (d_2 (d_3-1) + d-d_3+1) t^2.$$

By Lemma \ref{lemma: E(-1)},  
$$c(\mathcal{E}_{\mathcal{C}'}(-1)) = 1 + (-1-d_2 - d_3)t + (d_2 d_3 + d+1) t^2.$$

Compute again \eqref{eq:c_2}
for the exact sequence \eqref{eq:sheaves_ex_seq_general}, and the equality $d= \deg(\mathcal{C}') - | \mathcal{C}''| - \epsilon(\mathcal{C}',L)$ follows.
\\
\end{proof}

\begin{theorem}
\label{thm:del_inv}
Let $\mathcal{C}'$ be a plus-one generated curve with the exponents $(d_2, d_3)$ and level $d=\deg(\mathcal{C}')-|\mathcal{C}''| - \epsilon(\mathcal{C}',L), \mathcal{C}''\coloneqq  \mathcal{C}' \cap L$ for some line $L \subset \mathbb{P}^2$ such that $L$ is not included in  $\mathcal{C}'$. Then $\mathcal{C} := \mathcal{C}' \cup L$ is free with the exponents $(d_2, d_3)$.
\end{theorem}

\begin{proof}
Using Proposition \ref{prop:chern_2} and Lemma \ref{lemma: E(-1)}, we get that 
$$c(\mathcal{E}_{\mathcal{C}'}(-1)) = 1 + (-1-d_2 - d_3)t + (d_2 d_3 + d+1) t^2.$$

A computation of the second Chern number of $\mathcal{E}_{\mathcal{C}}$, using \eqref{eq:c_2}
and \eqref{eq:sheaves_ex_seq_general} and taking into account that $d=\deg(\mathcal{C}')-|\mathcal{C}''| - \epsilon(\mathcal{C}',L)$, shows that 
$$c_2(\mathcal{E}_{\mathcal{C}}) = d_2 d_3$$

By Theorem \ref{thm:VB_split}, to prove that $\mathcal{C}$ is free with the exponents $(d_2, d_3)$, it would be enough to show that the generic splitting type of $\mathcal{E}_{\mathcal{C}}$ onto a projective line is $(d_2, d_3)$.\\

Assume without loss of generality $d_2 \leq d_3$. Let $(d_2^K, d_3^K), \, d_2^K \leq  d_3^K,$ be the splitting type of $\mathcal{E}_{\mathcal{C}}$  onto some generic line $K \subset \mathbb{P}^2$. In any case, from Theorem \ref{thm:VB_split} it follows that $$d_2 d_3 \geq d_2^K d_3^K$$ Since $d_2 + d_3 = d_2^K + d_3^K$, it follows that $d_2^K  \leq d_2$. From \cite[Proposition 3.2 (2)-(4)]{AD} (see the proof of Proposition \ref{prop:chern_2} for details) we know that $d_2^K = mdr(\mathcal{C})$ except possibly when $\deg(\mathcal{C})=2m$ and  $d_2^K = m-1$ or  $\deg(\mathcal{C})=2m+1$ and $d_2^K = m$. We will treat separately these cases.\\

\noindent For now, assume $d_2^K = mdr(\mathcal{C})$.
Since $mdr(\mathcal{C}') = d_2$, then $mdr(\mathcal{C}) \in  \{d_2, d_2 +1 \}$. But $mdr(\mathcal{C}) = d_2+1$ is not possible, since $mdr(\mathcal{C}) = d_2^K \leq d_2$. So $mdr(\mathcal{C}) = d_2 = d_2^K$. But this means that the generic splitting type of  $\mathcal{E}_{\mathcal{C}}$ onto a line is $(d_2, d_3)$, and we are done.\\

Case $\deg(\mathcal{C})=2m, \; d_2^K = m-1$. Notice that $d_2+d_3 = 2m-1$, so $d_2 \leq m-1$. But $d_2 \geq d_2^K = m-1$, so $d_2 = m-1$. It follows $(d_2^K, d_3^K) = (d_2, d_3)$, and we are done.\\

 Case $\deg(\mathcal{C})=2m+1, \; d_2^K = m$. Notice that $d_2+d_3 = 2m$, so $d_2 \leq m$. In fact, $d_2=m$ since $d_2 \geq d_2^K$. Then  $(d_2^K, d_3^K) = (d_2, d_3)$, and this concludes the proof.
\end{proof}

\begin{theorem}
\label{thm_add_del_converse_free}
Let $(\mathcal{C}, \mathcal{C}', \mathcal{C}'')$ be a triple with respect to a line $L \subset \mathbb{P}^2$, such that $|\mathcal{C}''|+ \epsilon(\mathcal{C}', L) = d_2+1$. Then the following are equivalent:
\begin{enumerate}
\item[(1)] $\mathcal{C}'$ is free of exponents $(d_2, d_3)$.
\item[(2)] $\mathcal{C}$ is free of exponents $(d_2, d_3+1)$.
\end{enumerate}
\end{theorem}

\begin{proof}
$(1)  \Rightarrow (2)$ Assume $\mathcal{C}'$ is free with the exponents $(d_2, d_3)$. We already know that  $\mathcal{C} =  \mathcal{C}' \cup L$ must be either free or plus-one generated. If it is free, then, since  $|\mathcal{C}''|+ \epsilon(\mathcal{C}', L) = d_2+1$, its exponents are  $(d_2, d_3+1)$ by \cite[Theorem 1.3]{Dimca}.

If $\mathcal{C}$ would be plus-one generated, then its exponents would be $(d_2+1, d_3+1)$ and level $d$ such that  $|\mathcal{C}''|+ \epsilon(\mathcal{C}', L) = d+1$. But then $d=d_2$, so $d < d_2+1$, i.e., the level is strictly smaller than one of the exponents of the plus-one generated $\mathcal{C}$, a contradiction.\\

$(2)  \Rightarrow (1)$ Assume $\mathcal{C}$ is free of exponents $(d_2, d_3+1)$. We already know that  $\mathcal{C}'$ must be either free or plus-one generated. If it is free, then, since  $|\mathcal{C}''|+ \epsilon(\mathcal{C}', L) = d_2+1$, by Theorem \ref{thm:exp_computation}, it must be free with the 
exponents $(d_2, d_3)$. 
 
 If $\mathcal{C}'$ would be plus-one generated  plus-one generated, then, again by  Theorem \ref{thm:exp_computation}, its exponents would be $(d_2, d_3)$ and its level $d= \deg(\mathcal{C}') - |\mathcal{C}''|-  \epsilon(\mathcal{C}', L)$. Then $d = d_3 - 1 < d_3$, contradiction.
 \end{proof}
 
 \begin{theorem}
 \label{thm:char_free_deletion}
 Let $\mathcal{C}$ be a reduced curve such that $\mathcal{C} = \mathcal{C}' \cup L$ for some line $L \subset \mathbb{P}^2, \; L \not\subset \mathcal{C}'$.The following are equivalent:
\begin{enumerate}
    \item[1.] $\mathcal{C}$ is free with exponents $(d_2, d_3)$.
    \item[2.] $\mathcal{C}'$ is either 
    \begin{enumerate}
      \item[(i)]  free with  $\exp(\mathcal{C}') = (d_2, d_3-1)$ and $|\mathcal{C}' \cap L|  + \epsilon(\mathcal{C}',L)= d_2+1$ or 
      \item[(ii)] plus-one generated with  $\exp(\mathcal{C}') = (d_2, d_3)$ and level $d= \deg(\mathcal{C}) -1 - |\mathcal{C}' \cap L| - \epsilon(\mathcal{C}',L)$.
    \end{enumerate}
  \end{enumerate}
 \end{theorem}
 
 \begin{proof}
 $1.  \Rightarrow 2.$  By \cite[Theorem 1.1]{Dimca} and Theorem \ref{thm:exp_computation}.\\
 $2.  \Rightarrow 1.$ If (i) holds, then the claim follows from Theorem \ref{thm_add_del_converse_free}.   If (ii) holds, then the claim follows from Theorem \ref{thm:del_inv}.
 \end{proof}
 
  \begin{theorem}
 \label{thm:char_free_addition}
  Let $\mathcal{C}$ be a reduced curve. The following are equivalent:
  \begin{enumerate}
  \item[1.] $\mathcal{C}$ is free with exponents $(d_2, d_3)$.
   \item[2.]  For some line $L \subset \mathbb{P}^2$ that is not an irreducible component of $\mathcal{C}$, $\mathcal{C} \cup L$ is either
   \begin{enumerate}
 	  \item[(i)]  free with exponents  $(d_2, d_3+1)$ and  $| \mathcal{C} \cap L|+   \epsilon(\mathcal{C}, L) = d_2+1$  or
  	  \item[(ii)]   plus-one generated with exponents $(d_2+1, d_3+1)$ and level $d =   \epsilon(\mathcal{C}, L) + | \mathcal{C} \cap L|-1$.
   \end{enumerate}
    \item[3.]  For any line $L \subset \mathbb{P}^2$ that is not an irreducible component of $\mathcal{C}$, $\mathcal{C} \cup L$ is either
   \begin{enumerate}
 	  \item[(i)]  free with exponents  free with exponents  $(d_2, d_3+1)$ and  $| \mathcal{C} \cap L|+   \epsilon(\mathcal{C}, L) = d_2+1$  or
  	  \item[(ii)]   plus-one generated with exponents $(d_2+1, d_3+1)$ and level $d = \epsilon(\mathcal{C}, L) + | \mathcal{C} \cap L|-1$.
   \end{enumerate}
  \end{enumerate}
 \end{theorem}
 
 \begin{proof}
  $1.  \Rightarrow 3.$  follows from Theorem \ref{thm:add}.  $3.  \Rightarrow 2.$ is obvious. Let us prove  $2.  \Rightarrow 1.$ If (i) holds, then the claim follows from Theorem  \ref{thm_add_del_converse_free}.   If (ii) holds, then the claim follows from Theorem \ref{thm:add_converse}.
 \end{proof}
 
\begin{remark}
\label{rem:thm1.11Abe} The above Theorems  \ref{thm:char_free_deletion} and \ref{thm:char_free_addition} generalize \cite[Theorem 1.11]{A0}, where the case of projective line arrangements is treated.
\end{remark}

\bigskip

Anca~M\u acinic,
Simion Stoilow Institute of Mathematics of the Romanian Academy, 
P.O. Box 1-764, RO-014700 Bucharest, Romania. \\
\nopagebreak
\textit{E-mail address:} \texttt{anca.macinic@imar.ro}

\bigskip
Piotr Pokora,
Department of Mathematics,
University of the National Education Commission Krakow,
Podchor\c a\.zych 2,
PL-30-084 Krak\'ow, Poland. \\
\nopagebreak
\textit{E-mail address:} \texttt{piotr.pokora@up.krakow.pl}
\end{document}